\documentclass[11pt]{article}

\usepackage[letterpaper,margin=1in]{geometry}
\usepackage[T1]{fontenc}
\usepackage{amsmath,amssymb,amsthm,mathtools,bm,graphicx,url,microtype,setspace,needspace}
\usepackage{eso-pic}
\usepackage{titling}
\usepackage[hidelinks]{hyperref}
\predate{\begin{center}\small}\postdate{\par\end{center}}
\newtheoremstyle{upright}{0.45\baselineskip}{0.45\baselineskip}
  {\normalfont}{}{\itshape}{:}{0.5em}{}
\theoremstyle{upright}
\newtheorem{theorem}{Theorem}

\AddToHook{env/theorem/before}{\Needspace{6\baselineskip}}
\AddToHook{env/thebibliography/begin}{%
  \small\setstretch{1.02}\setlength{\itemsep}{0.2em}\setlength{\parskip}{0pt}%
}

\makeatletter
\newcommand{\bstctlcite}[1]{%
  \@bsphack
  \if@filesw\immediate\write\@auxout{\string\citation{#1}}\fi
  \@esphack
}
\makeatother

\newcommand{\R}{\mathbb{R}}
\newcommand{\E}{\mathbb{E}}
\newcommand{\tr}{\operatorname{tr}}

\newcommand{\cW}{\mathcal{W}}

\newcommand{\QED}{\mbox{\rule[0pt]{1.3ex}{1.3ex}}}
\let\articleappendix\appendix
\renewcommand{\appendix}{%
  \articleappendix
  \section*{Appendix}%
  \renewcommand{\thesubsection}{\Alph{subsection}}%
}

\title{\bfseries
Distributionally Robust Regret Optimal LQR with Common Stage-Law Ambiguity
}

\author{
Lukas-Benedikt Fiechtner\textsuperscript{1}\qquad Jose Blanchet\textsuperscript{1,2}\\[0.6em]
{\small\textsuperscript{1}Institute for Computational and Mathematical Engineering}\\
{\small\textsuperscript{2}Department of Management Science and Engineering}\\
{\small Stanford University, Stanford, CA 94305, USA}
}
\date{A version of this paper has been accepted for publication at the\\
65th IEEE Conference on Decision and Control (CDC 2026)}

\begin{document}
\bstctlcite{BSTcontrol}

\maketitle
\AddToShipoutPictureFG*{%
  \AtPageLowerLeft{%
    \hspace*{\dimexpr1in+\oddsidemargin\relax}%
    \raisebox{18pt}[0pt][0pt]{%
      \parbox[b]{\textwidth}{\normalfont\footnotesize
        \copyright\ 2026 IEEE. Personal use of this material is permitted.
        Permission from IEEE must be obtained for all other uses, in any current or future media,
        including reprinting/republishing this material for advertising or promotional purposes,
        creating new collective works, for resale or redistribution to servers or lists,
        or reuse of any copyrighted component of this work in other works.}%
    }%
  }%
}
\thispagestyle{empty}
\pagestyle{plain}

\begin{abstract}
We study what is, to our knowledge, the first tractable multistage ex-ante 
distributionally robust regret optimization (DRRO) formulation for stochastic 
control. We consider finite-horizon LQR with common stage-law ambiguity, 
where disturbances are independent across time but drawn from the same 
unknown stage law whose mean and covariance lie in a Gelbrich ball around 
nominal moments. Unlike the benign single-stage quadratic setting,
 the nominal controller is generally not regret-optimal: reuse of the
 stage law makes past disturbances informative for future decisions. 
 Despite the general hardness of DRRO, we show that, over affine 
 disturbance-feedback policies, the multistage DRRO-LQR problem 
 admits an exact semidefinite programming reformulation. An optimal 
 controller in this class is the nominal LQR controller plus a strictly causal 
 empirical-mean correction. We also characterize worst-case moment pairs 
 and show that, for the DRRO-optimal policy, they are not unique. 
 Portfolio liquidation experiments show that DRRO substantially reduces worst-case
 regret relative to DRO and the nominal controller, with comparatively
 modest increases in worst-case cost, and exhibits a learning effect:
 its correction matrices empirically approach the corresponding
 coefficients of the oracle controller that knows the true disturbance law in hindsight.

\end{abstract}

\section{Introduction}\label{sec:introduction}

Modern control is often designed from estimated disturbance models rather than a known 
ground-truth law. In finite-horizon stochastic control, the nominal prescription is 
to optimize expected cost under a reference model, which in LQR yields the nominal 
controller. This benchmark can be fragile under distributional misspecification. 
A standard remedy is distributionally robust optimization (DRO), which minimizes 
worst-case expected cost over an ambiguity set around the reference law 
\cite{mohajerin_esfahani_data-driven_2018}. While robust and often tractable, 
DRO can be conservative because it protects against the worst distribution in 
the ambiguity set without accounting for proximity to the model-wise optimum.

Distributionally robust regret optimization (DRRO) replaces worst-case cost by 
worst-case regret. In the \emph{ex-ante} notion studied here, regret is the 
expected performance loss of a policy relative to the best causal controller 
for the candidate law, followed by worst-case evaluation over the ambiguity set. 
This differs from the \emph{ex-post} regret formulations that dominate the control 
literature, where one compares a causal controller with a clairvoyant noncausal 
controller that knows the realized disturbance trajectory, then takes expectation, 
and only then the supremum. Thus ex-ante regret is model-wise and causal, 
whereas ex-post regret is pathwise and noncausal. Ex-ante DRRO has mostly 
been studied in static settings 
\cite{perakis_regret_2008,chen_regret_2021,agarwal_minimax_2022,cho_wasserstein_2024}.

The distinction between ex-post and ex-ante matters because the existing Wasserstein 
DRRO control literature is largely ex-post. Full-information finite-horizon, 
partially observed, infinite-horizon, and output-feedback variants have been 
studied in \cite{taha_distributionally_2023,hajar_wasserstein_2023,kargin_wasserstein_2023,yan_scherer_2025}. 
These papers are close in spirit, but differ from ours in two key ways. 
First, regret is measured relative to a clairvoyant noncausal controller
 with disturbance preview. Second, the Wasserstein ambiguity is typically 
 imposed on the joint law of the uncertainty process rather than on a single 
 stage law reused across time. 

To place our comparison baseline in context, we also recall the related DRO 
control literature. On the DRO side, tractable Wasserstein LQ/LQG models usually 
rely on stagewise or rectangular ambiguity. In \cite{NEURIPS2023_3b7a66b2} 
the authors study partially observed distributionally robust LQG with separate 
ambiguity sets for the initial condition, process noise, and measurement noise. 
There, the adversary may vary distributions across time, but does not adapt to 
the current state. A second line, including 
\cite{xu_mannor_2012,yang_convex_2017,yang_wasserstein_2021}, uses rectangular 
or statewise ambiguity to recover Bellman equations. In particular, \cite{kim_yang_2023} 
derive a Riccati recursion for a minimax LQ problem with a soft Wasserstein penalty,
 but their adversary is substantially stronger than ours because it selects a 
 worst-case distribution as a function of the current state at each stage.

Our setting differs from both strands above. We study \emph{common stage-law ambiguity}: 
the adversary selects a single disturbance law and reuses it across the horizon. 
This shared-law ambiguity is non-rectangular, so standard dynamic programming is unavailable.
 But that coupling is exactly what makes the problem interesting: 
 because the same law governs all disturbances, past observations are informative 
 about future ones. This repeated-use structure is reminiscent of control with a 
 fixed unknown parameter throughout time \cite{pmlr-v168-gurevich22a,Carruth2023OptimalAC},
  although here the uncertainty is distributional rather than parametric.
Additionally, in many applications, a fixed but unknown disturbance law is more realistic than
allowing an adversary to change the law with the state or controls, as in rectangular formulations.

Building on the static ex-ante Wasserstein DRRO analysis of \cite{fiechtner2025wasserstein},
 we study what is, to our knowledge, the first tractable multistage \emph{ex-ante} 
 DRRO formulation for stochastic control, and the first exact tractable dynamic quadratic 
 DRRO model in this sense. The static theory shows two seemingly conflicting facts: general 
 DRRO is NP-hard, yet in the static quadratic case the nominal decision is DRRO-optimal for 
 every ambiguity radius. This leaves a natural question: does DRRO for finite-horizon LQR remain 
 equally benign, or does the multistage structure create a genuinely new synthesis problem?
We answer this question by showing that, under common stage-law ambiguity, the nominal 
controller is generally no longer regret-optimal once there is more than one stage. 
Realized disturbances reveal information about the shared mean, and future controls 
can exploit that information.

Our main results are as follows. First, despite the general hardness of DRRO, we derive an
exact semidefinite reformulation of the DRRO-LQR problem over affine disturbance-feedback policies.
Second, the 
 optimal affine controller has a transparent structure: it is the nominal LQR policy plus 
 a strictly causal correction driven by the empirical mean of past disturbances, so the 
 history enters only through a learned summary statistic. Third, we derive a tractable DRO counterpart
 under the same shared-law Gelbrich ambiguity set and affine policy class. The common-law restriction
 already reduces conservatism relative to corresponding rectangular formulations, yielding a fairer
 benchmark that isolates the effect of the regret objective. Fourth, we characterize worst-case distributions, including the nonuniqueness 
 of the worst-case law for the DRRO-optimal policy. Finally, portfolio liquidation experiments show that DRRO
 substantially reduces worst-case regret relative to DRO and the nominal controller, with
 comparatively modest absolute increases in worst-case cost, and provide empirical evidence for a learning effect: along growing horizons,
 the DRRO correction coefficients approach their known-law counterparts in the tested example, 
 whereas the DRO coefficients remain separated.

\section{Problem Formulation}\label{sec:problem_formulation}

We consider a finite-horizon linear-quadratic control problem with dynamics and cost
\begin{align}
x_{t+1}&=A_tx_t+B_tu_t+\Xi_tw_t,\quad t=0,\dots,T-1,\label{eq:dynamics}\\
L&=\sum_{t=0}^{T-1}\left(x_t^\top Q_tx_t+2x_t^\top D_tu_t+u_t^\top R_tu_t\right)+x_T^\top Q_Tx_T,\label{eq:cost}
\end{align}
where $x_t\in\R^n$, $u_t\in\R^m$, and $w_t\in\R^d$ denote the state, control, and disturbance, respectively, and the matrices have appropriate dimensions with $R_t$ and $Q_t$ symmetric.

Standard stochastic LQR assumes i.i.d. disturbances with a known distribution $\hat{\mathbb P}$,
whose mean $\hat\mu$ and covariance $\hat\Sigma$ suffice to compute an optimal policy.
Because the true disturbance distribution is rarely known exactly, we retain the i.i.d. assumption
but let the unknown common law $\mathbb P$ belong to the \emph{uncertainty set}
\begin{equation}\label{eq:law_ambiguity}
\mathcal P_{\cW}:=\left\{\mathbb P:
\bigl(\E_{\mathbb P}[w],\operatorname{Cov}_{\mathbb P}(w)\bigr)\in\cW\right\},
\end{equation}
where $\mathbb P$ ranges over laws on $\R^d$ with finite second moments and
$\cW$ is the \emph{Gelbrich ball} of radius $\delta\ge0$
centered at the \emph{nominal moments} $(\hat\mu,\hat\Sigma)$:
\begin{equation}
\cW:=\left\{(\mu,\Sigma):\Sigma\succeq0,\quad
\mathsf G\bigl((\mu,\Sigma),(\hat\mu,\hat\Sigma)\bigr)\le\delta\right\},
\end{equation}
where the \emph{Gelbrich distance} is
\begin{equation}
\mathsf G\bigl((\mu,\Sigma),(\hat\mu,\hat\Sigma)\bigr)
:=\sqrt{\|\mu-\hat\mu\|_2^2+\mathsf B^2(\Sigma,\hat\Sigma)},
\end{equation}
and its covariance contribution is the squared \emph{Bures distance},
\begin{equation}
\mathsf B^2(\Sigma,\hat\Sigma)
=\tr\!\left(\Sigma+\hat\Sigma
-2(\hat\Sigma^{1/2}\Sigma\hat\Sigma^{1/2})^{1/2}\right).
\end{equation}

At time $t$, the controller chooses $u_t=\pi_t(w_0,\dots,w_{t-1})$ through a measurable map
$\pi_t:\R^{dt}\to\R^m$. The collection $\pi=(\pi_t)_{t=0}^{T-1}$ defines a causal policy,
whose controls we require to be square-integrable under every law in $\mathcal P_{\cW}$.
We assume that the states are fully observed and that each $\Xi_t$ has full column rank, so past disturbances can be reconstructed from observed state transitions and applied controls.
Let $\Pi_{\mathrm c}$ denote the set of all such causal policies.
For tractability of our regret formulation, we consider its subset of affine policies

\begin{equation}\label{eq:linear_policies}
\Pi_{\mathrm{aff}}^w:=
\left\{\pi\in\Pi_{\mathrm c}:u_t=g_t+\sum_{s<t}F_{ts}w_s\right\}.
\end{equation}
Here $g_t\in\R^m$ and $F_{ts}\in\R^{m\times d}$ are deterministic parameters.

For a fixed law $\mathbb P$, define the expected total cost of $\pi$ and the minimum expected total cost over causal policies by
\begin{equation}
J(\pi;\mathbb P):=\E_{\mathbb P^{\otimes T}}^\pi[L],\qquad
J^\star(\mathbb P):=\inf_{\tilde\pi\in\Pi_{\mathrm c}}J(\tilde\pi;\mathbb P).
\end{equation}
The expectation averages $L$ over $(w_0,\dots,w_{T-1})\sim\mathbb P^{\otimes T}$, with controls chosen by $\pi$ and states generated by the dynamics.
We define the \emph{ex-ante regret} of a policy $\pi$ under $\mathbb P$ as the difference between its expected total cost and that of an optimal causal controller that knows $\mathbb P$:
\begin{equation}\label{eq:regret_definition}
\mathcal R(\pi;\mathbb P):=J(\pi;\mathbb P)-J^\star(\mathbb P).
\end{equation}
Our goal is to find an affine controller that minimizes worst-case regret over all distributions in the uncertainty set $\mathcal P_{\cW}$:
\begin{equation*}\tag{DRRO}\label{eq:drro_problem}
\inf_{\pi\in\Pi_{\mathrm{aff}}^w}\ \sup_{\mathbb P\in\mathcal P_{\cW}}
\mathcal R(\pi;\mathbb P).
\end{equation*}

As a benchmark, we also consider an affine DRO controller that minimizes worst-case expected total cost over the same uncertainty set:
\begin{equation*}\tag{DRO}\label{eq:dro_problem}
\inf_{\pi\in\Pi_{\mathrm{aff}}^w}\ \sup_{\mathbb P\in\mathcal P_{\cW}}
J(\pi;\mathbb P).
\end{equation*}
Using the same policy class and common-law ambiguity for DRO and DRRO isolates the effect of the regret objective.
In both formulations we restrict to affine policies for tractability and claim optimality only within this class for $\delta>0$.

DRO may sacrifice substantial performance under favorable laws for only a slight improvement in worst-case cost.
By measuring performance relative to each law's optimal causal controller, DRRO accounts for both avoidable losses and missed opportunities.

\section{SDP Reformulations}\label{sec:drro_reformulation}

\subsection{Known-Law LQR}

To solve \eqref{eq:drro_problem}, we first solve the classical LQR problem for a fixed law $\mathbb P$, 
expressing the benchmark $J^\star(\mathbb P)$ in terms of its mean and covariance.
Let $V_t(x;\mu,\Sigma)$ denote the minimum expected remaining cost from state $x$ at time $t$ under $\mathbb P$
which depends only on the moment pairs $(\mu,\Sigma)$. Bellman recursion gives
\begin{equation}\label{eq:affine_value_function}
V_t(x;\mu,\Sigma)=x^\top S_tx+2x^\top P_t\mu+\mu^\top N_t\mu+\tr(\Gamma_t\Sigma).
\end{equation}
where the coefficients $S_t$, $P_t$, $N_t$, and $\Gamma_t$ satisfy the backward recursion
\begin{align}
S_t&=Q_t+A_t^\top S_{t+1}A_t-C_t^\top M_t^{-1}C_t,\\
P_t&=(A_t+B_tK_t)^\top(S_{t+1}\Xi_t+P_{t+1}),\label{eq:def_P}\\
N_t&=N_{t+1}+\Xi_t^\top S_{t+1}\Xi_t+P_{t+1}^\top\Xi_t+\Xi_t^\top P_{t+1}-\bar H_t^\top M_t\bar H_t,\label{eq:def_N}\\
\Gamma_t&=\Gamma_{t+1}+\Xi_t^\top S_{t+1}\Xi_t.\label{eq:def_Gamma}
\end{align}
initialized at $S_T=Q_T$ and $P_T=N_T=\Gamma_T=0$ where
\begin{align}
M_t&:=R_t+B_t^\top S_{t+1}B_t,\label{eq:def_M}\\
C_t&:=B_t^\top S_{t+1}A_t+D_t^\top,\\
K_t&:=-M_t^{-1}C_t,\label{eq:def_K}\\
\bar H_t&:=-M_t^{-1}B_t^\top(S_{t+1}\Xi_t+P_{t+1}).\label{eq:def_Hbar}
\end{align}
We assume $M_t\succ0$ at every stage, ensuring a well-defined recursion while allowing indefinite stage and terminal costs.

The optimal causal control and cost are
\begin{equation}\label{eq:known_law_optimum}
u_t^\star=K_tx_t+\bar H_t\mu,
\qquad J^\star(\mu,\Sigma)=V_0(x_0;\mu,\Sigma).
\end{equation}
We write $J^\star(\mu,\Sigma):=J^\star(\mathbb P)$ since the benchmark depends on the law only through its moments. For an affine policy $\pi$ with coefficients $g=(g_t)$ and $F=(F_{ts})$ as in \eqref{eq:linear_policies}, we likewise write $J(g,F;\mu,\Sigma):=J(\pi;\mathbb P)$.
In particular, the optimal policy for the nominal measure is 
\begin{equation}\label{eq:policy_nominal}
  u^{\mathrm{nom}}_t = K_tx_t + \bar H_t \hat \mu
\end{equation}

\subsection{Distributionally Robust Regret Optimization}\label{sec:policy_interpretation_drro}

The following theorem gives an SDP for computing a DRRO-optimal affine policy, using the matrices $M_t$ and $\bar H_t$ obtained by solving the nominal LQR problem.

\begin{theorem}\label{thm:drro_sdp}
Assume $\hat\Sigma\succ0$ and $M_t\succ0$ for all $t$. For $\delta>0$, problem \eqref{eq:drro_problem} 
has the same optimal value as the following SDP, which attains its minimum:
\begin{equation}\label{eq:convex_row_sum_sdp}
\begin{aligned}
&\min_{\substack{\Lambda_1,\dots,\Lambda_{T-1},\,\gamma,\,U,\\V_1,\dots,V_{T-1},\,W_1,\dots,W_{T-1}}}
\quad
\gamma\bigl(\delta^2-\tr(\hat\Sigma)\bigr)+\tr(U\hat\Sigma)\\
\text{s.t.}\quad
&\gamma I_d \succeq \bar H_0^\top M_0\bar H_0 + \sum_{t=1}^{T-1} W_t,\qquad
\begin{pmatrix}
\gamma I_d-\sum_{t=1}^{T-1}V_t & \gamma I_d\\
\gamma I_d & U
\end{pmatrix}\succeq 0,\\
&
\begin{pmatrix}
t\,M_t^{-1} & \Lambda_t\\
\Lambda_t^\top & V_t
\end{pmatrix}\succeq 0,\qquad
\begin{pmatrix}
M_t^{-1} & \Lambda_t-\bar H_t\\
(\Lambda_t-\bar H_t)^\top & W_t
\end{pmatrix}\succeq 0,
\quad t=1,\dots,T-1,\\
&\Lambda_t\in\R^{m\times d},\quad U,V_t,W_t\in\mathbb S^d.
\end{aligned}
\end{equation}
Any optimal $\Lambda$ yields an optimal affine controller
\begin{equation}\label{eq:policy_representation_drro}
u^{\mathrm{DRRO}}_0 = K_0 x_0 + \bar H_0\hat\mu,
\qquad
u^{\mathrm{DRRO}}_t = K_t x_t + \bar H_t\hat\mu + \Lambda_t(\bar w_{0:t-1} - \hat\mu)
\end{equation}
where $\bar w_{0:t-1} = \frac{1}{t}\sum_{s=0}^{t-1}w_s$ for $t=1,\dots,T-1$.
\end{theorem}

The proof is given in Subsection~\ref{app:drro_proof} of the Appendix. Theorem~\ref{thm:drro_sdp} expresses the DRRO-optimal policy as
the nominal policy \eqref{eq:policy_nominal} plus the correction $\Lambda_t(\bar w_{0:t-1}-\hat\mu)$.
Under a true stage law with mean $\mu$, the empirical deviation $\bar w_{0:t-1}-\hat\mu$ estimates the error $\mu-\hat\mu$ in the nominal mean.
If, along growing horizons with $t\to\infty$, the correction matrices $\Lambda_t$ remain bounded and $\|\Lambda_t-\bar H_t\|\to0$,
the law of large numbers implies that the DRRO action approaches the known-law action $K_tx_t+\bar H_t\mu$ at the same state almost surely.
This conditional convergence holds under any fixed i.i.d. law with finite second moments, even when its moments lie outside $\cW$.
We present numerical evidence for coefficient convergence in Section~\ref{sec:experiments} and defer its proof to a journal version.
For $T=1$, the nominal controller is DRRO-optimal, as in the quadratic Wasserstein setting of \cite[Theorem~4]{fiechtner2025wasserstein}.

\Needspace{10\baselineskip}
\subsection{Distributionally Robust Optimization}

Using the same LQR recursion, we obtain an SDP for computing a DRO-optimal affine policy.

\begin{theorem}\label{thm:dro_sdp}
Under the assumptions of Theorem~\ref{thm:drro_sdp}, set
\begin{equation}\label{eq:dro_affine_constants}
a_0=x_0^\top S_0x_0+2x_0^\top P_0\hat\mu+\hat\mu^\top N_0\hat\mu,
\qquad b_0=P_0^\top x_0+N_0\hat\mu.
\end{equation}
For $\delta>0$, problem \eqref{eq:dro_problem} has the same optimal value as the following SDP, which attains its minimum:
\begingroup
\allowdisplaybreaks[1]
\begin{align}
&\min_{\substack{\Lambda_1,\dots,\Lambda_{T-1},\,\gamma,\,\rho,\,U,\\V_1,\dots,V_{T-1},\,W_1,\dots,W_{T-1}}}
\quad a_0+\rho+\gamma(\delta^2-\tr\hat\Sigma)+\tr(U\hat\Sigma)\notag\\*
\text{s.t.}\quad
&\gamma\ge0,\quad
\begin{pmatrix}\gamma I_d-N_0&b_0\\b_0^\top&\rho\end{pmatrix}\succeq0,\qquad
\gamma I_d-N_0\succeq\bar H_0^\top M_0\bar H_0+\sum_{t=1}^{T-1}W_t,\notag\\
&\begin{pmatrix}
\gamma I_d-\Gamma_0-\sum_{t=1}^{T-1}V_t&\gamma I_d\\
\gamma I_d&U
\end{pmatrix}\succeq0,\notag\\
&\begin{pmatrix}
t\,M_t^{-1}&\Lambda_t\\
\Lambda_t^\top&V_t
\end{pmatrix}\succeq0,\qquad
\begin{pmatrix}
M_t^{-1}&\Lambda_t-\bar H_t\\
(\Lambda_t-\bar H_t)^\top&W_t
\end{pmatrix}\succeq0,
\quad t=1,\dots,T-1,\notag\\*
&\Lambda_t\in\R^{m\times d},\quad U,V_t,W_t\in\mathbb S^d.\label{eq:dro-F-reduced-problem}
\end{align}
\endgroup
Any optimal $(\Lambda,\gamma)$ yields an optimal affine controller
\begin{equation}\label{eq:dro_optimal_policy}
u_0^{\mathrm{DRO}}=K_0x_0+\bar H_0\theta,
\qquad u_t^{\mathrm{DRO}}=K_tx_t+\bar H_t\theta
+\Lambda_t(\bar w_{0:t-1}-\theta),\quad t\ge1,
\end{equation}
where $\theta=\hat\mu+(\gamma I_d-N_0)^\dagger b_0$ and $\dagger$ denotes the Moore--Penrose inverse. At $\delta=0$, $u^{\mathrm{nom}}$ is optimal.
\end{theorem}

The proof follows similar arguments to Theorem~\ref{thm:drro_sdp} and is omitted for brevity.
The policy \eqref{eq:dro_optimal_policy} uses the shifted center $\theta$ in place of $\hat\mu$, so its initial action generally differs from the nominal action.
If $\|\Lambda_t-\bar H_t\|\to0$ along growing horizons with $t\to\infty$ and both $\Lambda_t$ and $\theta$ remain bounded,
the law of large numbers would yield the same learning effect as in DRRO.
The numerical results in Section~\ref{sec:experiments} show a persistent gap between the DRO correction matrices and $\bar H_t$ in the tested example. A counterexample to coefficient convergence is deferred to a journal version.

\subsection{Computational Complexity of the SDPs}

For $\delta>0$, let $N_{\mathrm{var}}$ denote the number of independent scalar decision variables, counting a symmetric $d\times d$ matrix as $d(d+1)/2$ scalars. In \eqref{eq:convex_row_sum_sdp}, $N_{\mathrm{var}}=1+(T-1)md+(2T-1)\frac{d(d+1)}{2}$.
Let $k_j$ denote the order (number of rows) of the matrix in the $j$th positive semidefinite constraint. The block orders are $d$, $2d$, and $m+d$ repeated $2(T-1)$ times, giving $\nu:=\sum_j k_j=3d+2(T-1)(m+d)$.

Starting from suitably centered, strictly feasible primal and dual points, a standard interior-point method requires $O(\sqrt{\nu}\log(1/\varepsilon))$ iterations to reduce the duality gap by a factor $\varepsilon\in(0,1)$~\cite[Thm.~5.1]{vandenberghe_boyd1996}.
Using dense linear algebra while retaining the semidefinite block structure, each iteration costs $O\!\left(N_{\mathrm{var}}^3+N_{\mathrm{var}}^2\sum_j k_j^2+N_{\mathrm{var}}\sum_j k_j^3\right)$ arithmetic operations~\cite[Sec.~11.8.3]{boyd_vandenberghe2004}.
The resulting total bound for $T\ge2$ is
\begin{equation}\label{eq:sdp_solve_complexity}
O\!\left(T^{7/2}d^2(m+d)^{9/2}\log(1/\varepsilon)\right).
\end{equation}
The DRO SDP adds one scalar variable, a semidefinite block of order $d+1$, and $\gamma\ge0$, so it has the same asymptotic bound. The bound \eqref{eq:sdp_solve_complexity} is conservative because it does not exploit further sparsity in the Newton equations.

The nominal backward recursion and SDP data construction require at most $O(T(n+m+d)^3)$ arithmetic operations with dense linear algebra, including $O(Tm^3)$ to form all $M_t^{-1}$. Each inverse is computed once and reused throughout the SDP solve. At $\delta=0$, nominal LQR suffices.

\section{Worst-Case Moments and Laws}\label{sec:worst_case}

The following theorem characterizes worst-case laws for an affine policy parametrized as in \eqref{eq:policy_parametrization_new}.
Let $a,c,\mathcal A,\mathcal B$ denote the policy's regret coefficients
in \eqref{eq:regret_coefficients}, derived in Subsection~\ref{app:drro_proof} of the Appendix.
Set $\alpha:=\lambda_{\max}(\mathcal A)$ and $\beta:=\lambda_{\max}(\mathcal B)$.

\begin{theorem}\label{thm:wc_recipe}
Under the assumptions of Theorem~\ref{thm:drro_sdp} and $\delta>0$, if $\alpha=\beta=0$, the policy's worst-case regret equals $a$ and every law in $\mathcal P_{\cW}$ is worst-case.
Otherwise, its worst-case regret equals $\min_{\gamma\in\mathcal D}\Psi(\gamma)$, where
\begin{equation}\label{eq:wc_scalar_domain}
\mathcal D=\left\{\gamma\geq0:\gamma>\alpha,\ \gamma\geq\beta,\ 
c\in\operatorname{range}(\gamma I_d-\mathcal B)\right\},
\end{equation}
\begin{equation}\label{eq:wc_scalar_objective}
\Psi(\gamma)=a+\gamma(\delta^2-\tr\hat\Sigma)
+c^\top(\gamma I_d-\mathcal B)^\dagger c
+\gamma^2\tr\!\bigl((\gamma I_d-\mathcal A)^{-1}\hat\Sigma\bigr).
\end{equation}
The function $\Psi$ is convex and has a unique minimizer $\gamma^\star$. For $\gamma\in\mathcal D$, define
\begin{equation}\label{eq:wc_moment_maps}
z_\gamma:=(\gamma I_d-\mathcal B)^\dagger c,
\qquad \Sigma_\gamma:=\gamma^2(\gamma I_d-\mathcal A)^{-1}\hat\Sigma(\gamma I_d-\mathcal A)^{-1}.
\end{equation}
The worst-case regret laws are exactly those with moment pairs $(\hat\mu+z_{\gamma^\star}+v,\Sigma_{\gamma^\star})$ for some $v$ satisfying
\begin{equation}\label{eq:wc_moment_recovery}
v\in\ker(\gamma^\star I_d-\mathcal B),
\qquad \|v\|_2^2=\delta^2-\|z_{\gamma^\star}\|_2^2
-\mathsf B^2(\Sigma_{\gamma^\star},\hat\Sigma).
\end{equation}
The worst-case moment pair is unique exactly when the right-hand side of the norm condition in \eqref{eq:wc_moment_recovery} is zero.

Every DRRO-optimal affine policy admits at least two worst-case moment pairs and hence at least two worst-case laws.
For a DRRO-optimal policy with $\max\{\alpha,\beta\}>0$, $\gamma^\star=\beta>\alpha$ and the right-hand side of the norm condition in \eqref{eq:wc_moment_recovery} is strictly positive.
\end{theorem}

\section{Experiments}\label{sec:experiments}\label{sec:execution_experiment}

We consider a single-asset Almgren--Chriss-type execution model \cite{almgren_chriss_2001}, 
in which an agent liquidates an initial position $q_0>0$ by time $T$ in a market with permanent 
and instantaneous price impact. The state is two-dimensional, comprising remaining inventory 
$q_t$ and midprice $s_t$, whereas trades $u_t$ and price innovations $w_t$ are scalar 
($n=2$, $m=d=1$).
 With permanent impact $G>0$ and instantaneous impact $R>0$, 
 the dynamics and execution price are
\begin{equation}\label{eq:execution_dynamics}
q_{t+1}=q_t-u_t,\qquad s_{t+1}=s_t-Gu_t+w_t,
\qquad s_t^{\mathrm{exe}}=s_t-Ru_t,\qquad q_T=0.
\end{equation}
Permanent impact affects future midprices, whereas instantaneous impact changes the price 
received for the current trade. We allow purchases and short positions during liquidation.

The agent maximizes expected selling proceeds less an inventory-risk penalty based on the fixed nominal variance $\hat\sigma^2$:
\begin{equation}\label{eq:execution_proceeds}
\E_{\mathbb P^{\otimes T}}^\pi\!\left[\sum_{t=0}^{T-1}
\left(u_ts_t^{\mathrm{exe}}-\lambda\hat\sigma^2q_t^2\right)\right],
\end{equation}
where $\lambda>0$ measures risk aversion. Writing $p_t=s_t-s_0$ and using $\sum_tu_t=q_0$, we remove the constant reference-price proceeds $s_0q_0$ and obtain the equivalent expected-cost minimization with
\begin{equation}\label{eq:execution_cost}
L=\sum_{t=0}^{T-1}\left(\lambda\hat\sigma^2q_t^2+Ru_t^2-p_tu_t\right).
\end{equation}
The constraint $q_T=0$ forces $u_{T-1}=q_{T-1}$. Substituting this trade into \eqref{eq:execution_cost} leaves $T-1$ decision stages.
For $x_t=(q_t,p_t)^\top$ and $t=0,\dots,T-2$, the matrices in Section~\ref{sec:problem_formulation} are
\[
A_t\equiv I_2,\quad B_t\equiv\begin{pmatrix}-1\\-G\end{pmatrix},\quad
\Xi_t\equiv\begin{pmatrix}0\\1\end{pmatrix},\quad
Q_t\equiv\operatorname{diag}(\lambda\hat\sigma^2,0),\quad D_t\equiv\begin{pmatrix}0\\-1/2\end{pmatrix},\quad R_t\equiv R.
\]
The terminal cost is $x_{T-1}^\top Q_{T-1}x_{T-1}$, where
\begin{equation}\label{eq:execution_hard_terminal}
Q_{T-1}=\begin{pmatrix}\lambda\hat\sigma^2+R&-1/2\\-1/2&0\end{pmatrix}.
\end{equation}
A Schur-complement argument shows that $R-G/2>0$ ensures $M_t\succ0$, so the LQR recursion is well defined.

Inventories and trades are measured in millions of shares, prices in dollars per share, and costs in millions of dollars. We use $T=20$, $q_0=1$, $s_0=100$, $G=0.08$, $R=0.15$, $\lambda=2$, $\hat\mu=0$, $\hat\sigma=0.12$, and $\delta=0.10$. Thus $R-G/2=0.11>0$. In mean--standard-deviation coordinates, the Gelbrich ball is the disk
$(\mu-\hat\mu)^2+(\sigma-\hat\sigma)^2\le\delta^2$, $\sigma\ge0$.

The single-asset rows of Table~\ref{tab:execution_results} compare the nominal policy, DRO, and DRRO in terms of nominal expected cost, worst-case regret, and worst-case expected cost.
Nominal costs use the same law $\hat{\mathbb P}$ for all policies. Worst-case cost and regret are maximized 
separately for each policy, so the maximizing laws may differ across policies and objectives. 
DRRO substantially reduces worst-case regret at the cost of only comparatively modest absolute increases in worst-case and nominal costs.
\begin{table}[!htbp]
\centering
\caption{Single- and four-asset liquidation. Lower values are better.}
\label{tab:execution_results}
\setlength{\tabcolsep}{7pt}
\begin{tabular}{clrrr}
\hline
Assets & Policy & Nominal cost & Worst regret & Worst cost\\
\hline
1 & Nominal & \textbf{0.1125} & 1.3859 & 0.2642\\
 & DRO & 0.1166 & 1.5417 & \textbf{0.1166}\\
 & DRRO & 0.2236 & \textbf{0.4382} & 0.4904\\
\hline
4 & Nominal & \textbf{0.1106} & 3.1399 & 0.2933\\
 & DRO & 0.1150 & 3.2362 & \textbf{0.1151}\\
 & DRRO & 0.3366 & \textbf{0.9315} & 0.8916\\
\hline
\end{tabular}
\end{table}

Figure~\ref{fig:execution_cost_difference} shows the difference in expected cost between DRO and DRRO across the moment ball with positive values favoring DRRO.
DRRO improves on DRO over broad regions of positive and negative drift, 
whereas DRO is favored near zero drift at higher volatility. DRRO has lower cost at approximately $61\%$ of the feasible points on the grid.
Relative to DRO, DRRO's largest cost saving on the grid exceeds three times its largest cost increase.

\begin{figure}[!htbp]
\centering
\includegraphics[width=0.72\linewidth]{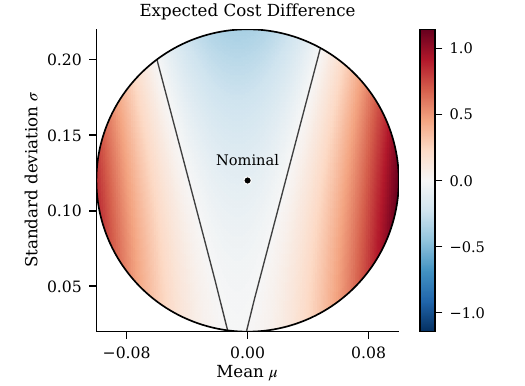}
\caption{Difference in expected cost between DRO and DRRO over the Gelbrich ball.}
\label{fig:execution_cost_difference}
\end{figure}

Figure~\ref{fig:execution_inventory_paths} compares expected inventories for the Nominal, DRRO, DRO, and Oracle policies under negative, zero, and positive drift, with $\sigma=\hat\sigma$.
Oracle knows the true stage law and uses the hindsight optimal causal controller in \eqref{eq:known_law_optimum}.
At $t=0$ before observing any price increments, DRRO uses the nominal drift estimate and makes the same initial sale as Nominal.
For $t\ge1$, however, DRRO uses past disturbances to adapt its liquidation schedule.
At the nominal drift $\mu=\hat\mu=0$, its expected inventory follows the nominal liquidation schedule.
Under negative drift, it sells faster, temporarily takes a short position, and buys back to zero inventory.
Under positive drift, it buys after its initial sales to benefit from expected price appreciation, then liquidates by the deadline.
Overall, DRRO follows Oracle's expected inventory path most closely.
In this experiment, Nominal and DRO follow deterministic liquidation schedules that do not adapt to the observed disturbances.

\begin{figure}[!htbp]
\centering
\includegraphics[width=\textwidth]{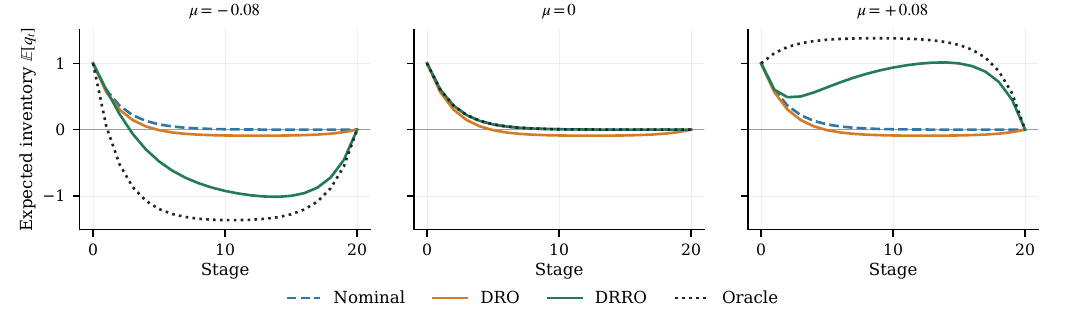}
\caption{Expected inventory in millions of shares under $\mu=-0.08$, $0$, and $0.08$, with $\sigma=0.12$.}
\label{fig:execution_inventory_paths}
\end{figure}

Figure~\ref{fig:execution_learning_gap} plots $|\Lambda_t-\bar H_t|$ at $t=T/2$ for increasing horizons with fixed per-period parameters.
Evaluating midway avoids the forced final trade, where both coefficients vanish.
The DRRO gap decreases, whereas the DRO gap and benchmark coefficient remain nearly constant, illustrating DRRO's learning effect in this example.

\begin{figure}[!htbp]
\centering
\includegraphics[width=0.6\linewidth]{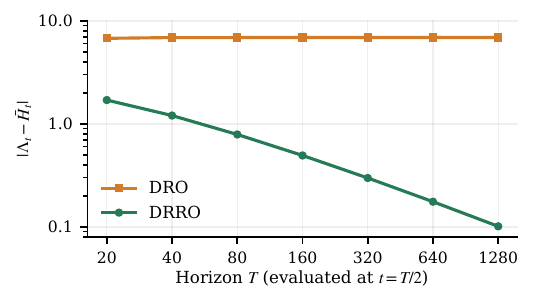}
\caption{Correction-coefficient gaps at $t=T/2$. Each horizon is solved with the same per-period parameters and ambiguity set.}
\label{fig:execution_learning_gap}
\end{figure}

The preceding example had scalar disturbances so that we could visualize the full Gelbrich moment ball in $(\mu,\sigma)$ coordinates (although the state is still two-dimensional).
To assess multidimensional disturbances, we extend \eqref{eq:execution_dynamics} to $d$ assets, with vector inventories, prices, trades, and innovations.
Off-diagonal entries of the impact matrices $G$ and $R$ describe \emph{cross-impact}: trading one asset affects another asset's midprice or execution price, respectively.
The cost becomes $\sum_t(\lambda q_t^\top\hat\Sigma q_t+u_t^\top R u_t-p_t^\top u_t)$, with $q_T=0$.
We keep $T=20$, $\lambda=2$, and $\delta=0.10$, use diagonal impact matrices $G=0.08I_d$, $R=0.15I_d$ (no cross-impact), and set $\hat\mu=0_d$,
$\hat\Sigma_{ij}=0.12^2(1/2)^{|i-j|}$, and $q_0=(1,0,\ldots,0)^\top$.
With four assets ($n=8$, $m=d=4$), Table~\ref{tab:execution_results} again shows substantially lower worst-case regret for DRRO, with modest absolute increases in nominal and worst-case costs.

For up to four assets at $T=20$, solving either robust SDP takes less than a second.
We use CVXPY and Clarabel with feasibility and gap tolerances $10^{-9}$.
Experiments were run on an Apple M2 machine with 8 CPU cores and 16 GB of memory.
For $d=1,2,4$, the DRRO SDPs have 56, 184, and 659 independent scalar variables, with conic solve times of $1.7$, $8.6$, and $62.3$ ms.
For $d=1,2,4$, DRO uses 57, 185, and 660 independent scalar variables, with conic solve times of $2.2$, $18.8$, and $122.1$ ms, respectively.
For $d=1$, $T=10,20,40$, DRRO uses 26, 56, and 116 variables, with conic solve times of $0.65$, $1.7$, and $4.7$ ms, respectively.

\section{Conclusion}

We derived an SDP for affine ex-ante DRRO-LQR under common-law ambiguity and characterized worst-case laws, which are nonunique for DRRO-optimal affine policies.
Experiments show substantially lower worst-case regret than DRO with comparatively modest absolute increases in worst-case cost, and illustrate DRRO's learning effect.
In a journal version, we plan to prove convergence of the correction coefficients to their known-law counterparts along growing horizons and investigate optimality beyond affine policies.
Extensions to partial observability are another potential direction for future research.

\section*{Acknowledgments}
We acknowledge the use of GPT-5.4, OpenAI, \cite{chatgpt2026}, for editorial assistance, writing code for experiments, 
and verifying mathematical derivations. We reviewed and edited all outputs and take full responsibility
for the scientific integrity of the final content.
J. Blanchet also gratefully acknowledges support from ONR under award N00014-24-1-2655, and the National Science Foundation (NSF) under grants 2312204 and 2403007.

\appendix
\subsection{Proof of Theorem~\ref{thm:drro_sdp}}\label{app:drro_proof}
Fix an affine policy $\pi$ with coefficients $(g,F)$ as in \eqref{eq:linear_policies} and a law $\mathbb P$ with moments $(\mu,\Sigma)$ for which we suppress the moment arguments in $V_t$.
Let $\ell_t(x,u)$ denote the stage cost in \eqref{eq:cost} and $\mathcal F_t=\sigma(w_0,\ldots,w_{t-1})$.
Set $\eta_t=u_t-K_tx_t-\bar H_t\mu$.
Completing the square in the Bellman recursion gives
\begin{equation}\label{eq:bellman_square_completion}
\ell_t(x_t,u_t)+\E^\pi\!\left[V_{t+1}(x_{t+1})\mid\mathcal F_t\right]-V_t(x_t)
=\eta_t^\top M_t\eta_t.
\end{equation}
Taking expectations in \eqref{eq:bellman_square_completion} and summing over time telescopes the value-function terms.
Since $V_T(x_T)=x_T^\top Q_Tx_T$ and $V_0(x_0)=J^\star(\mu,\Sigma)$, we obtain the advantage identity
\begin{equation}\label{eq:cdc-advantage}
J(g,F;\mu,\Sigma)-J^\star(\mu,\Sigma)
=\E_{\mathbb P^{\otimes T}}^\pi\!\left[\sum_{t=0}^{T-1}\eta_t^\top M_t\eta_t\right].
\end{equation}
This suggests to reparameterize the affine policies as
\begin{equation}\label{eq:policy_parametrization_new}
u_t=K_tx_t+\bar H_t\hat\mu
+\sum_{s<t}\tilde F_{ts}(w_s-\hat\mu)+\tilde g_t.
\end{equation}
Since $x_t$ is affine in past disturbances, this is a reparameterization of \eqref{eq:linear_policies}.
Set $z:=\mu-\hat\mu$. Independence of the disturbances gives the worst-case regret
\begin{equation}\label{eq:controller_objective_in_Fg}
\mathcal R(\tilde g,\tilde F)
=\max_{(\hat\mu+z,\Sigma)\in\cW}
\bigl\{a(\tilde g,\tilde F)+z^\top\mathcal B(\tilde F)z
+2c(\tilde g,\tilde F)^\top z+\tr(\mathcal A(\tilde F)\Sigma)\bigr\}
\end{equation}
where, with $\Lambda_t:=\sum_{s<t}\tilde F_{ts}$ and $\Lambda_0=0$, the coefficients are
\begin{align}
a(\tilde g,\tilde F)&=\sum_t\tilde g_t^\top M_t\tilde g_t,\nonumber\\
c(\tilde g,\tilde F)&=\sum_t(\Lambda_t-\bar H_t)^\top M_t\tilde g_t,\nonumber\\
\mathcal A(\tilde F)&=\sum_{t=1}^{T-1}\sum_{s<t}\tilde F_{ts}^\top M_t\tilde F_{ts},\quad
\mathcal B(\tilde F)=\sum_{t=0}^{T-1}(\Lambda_t-\bar H_t)^\top M_t(\Lambda_t-\bar H_t).\label{eq:regret_coefficients}
\end{align}

The moment ball is symmetric about $\hat\mu$, so averaging the objective at $z$ and $-z$ gives
\(
\mathcal R(\tilde g,\tilde F)\ge a(\tilde g,\tilde F)+\mathcal R(0,\tilde F).
\)
Since $M_t\succ0$, we have $a(\tilde g,\tilde F)>0$ whenever $\tilde g\ne0$. Hence every optimal policy has $\tilde g=0$.
For fixed $\Lambda_t$ and $t\ge1$,
\[
\sum_{s<t}\tilde F_{ts}^\top M_t\tilde F_{ts}
-\frac{1}{t}\Lambda_t^\top M_t\Lambda_t
=\sum_{s<t}\left(\tilde F_{ts}-\frac{\Lambda_t}{t}\right)^\top
M_t\left(\tilde F_{ts}-\frac{\Lambda_t}{t}\right)\succeq0.
\]
Thus, replacing $\tilde F_{ts}$ by $\Lambda_t/t$ leaves $\mathcal B(\tilde F)$ unchanged and decreases $\mathcal A(\tilde F)$ in semidefinite order, so it cannot increase worst-case regret. We may therefore restrict to
\(\tilde F_{ts}=\frac{\Lambda_t}{t}\).
With these reductions the problem reduces to minimizing
\begin{equation}\label{eq:regret_moment_representation}
\mathcal R(\Lambda):=
\max_{(\hat\mu+z,\Sigma)\in\cW}
\bigl\{z^\top\mathcal B(\Lambda)z+\tr(\mathcal A(\Lambda)\Sigma)\bigr\},
\end{equation}
over \(\Lambda=(\Lambda_t)_{t=1}^{T-1}\), where
\(\mathcal A(\Lambda)=\sum_{t=1}^{T-1}\frac{1}{t}\Lambda_t^\top M_t\Lambda_t\) and
\(\mathcal B(\Lambda)=\sum_{t=0}^{T-1}(\Lambda_t-\bar H_t)^\top M_t(\Lambda_t-\bar H_t)\).

Since $\cW$ is compact~\cite[Prop.~2]{nguyen_mean_covariance_2025} and the objective defining $\mathcal R(\Lambda)$ is continuous, $\mathcal R(\Lambda)$ is finite and its maximum is attained for every $\Lambda$.
Gelbrich~\cite[Prop.~2.2]{gelbrich1990} implies that
\begin{equation}\label{eq:gelbrich_representation}
\mathsf G^2\bigl((\mu,\Sigma),(\hat\mu,\hat\Sigma)\bigr)
=\begin{aligned}[t]
\min_{Y\in\R^{d\times d}}\quad&\|\mu-\hat\mu\|^2+\tr(\Sigma+\hat\Sigma-2Y)\\
\text{s.t.}\quad&
\begin{pmatrix}
\Sigma&Y\\
Y^\top&\hat\Sigma
\end{pmatrix}\succeq0,
\end{aligned}
\end{equation}
with
\(
Y=\hat\Sigma^{-1/2}
(\hat\Sigma^{1/2}\Sigma\hat\Sigma^{1/2})^{1/2}
\hat\Sigma^{1/2}
\)
attaining the minimum.
Set $\beta=\lambda_{\max}(\mathcal B(\Lambda))$.
Using \eqref{eq:gelbrich_representation} and the identity $\max_{\|z\|^2\le r}z^\top\mathcal B(\Lambda)z=\beta r$ for $r\ge0$, we rewrite $\mathcal R(\Lambda)$ as the following attained maximum:
\begin{equation}\label{eq:lifted_regret_primal}
\begin{aligned}
\max_{r,\Sigma,Y}\quad&\beta r+\tr(\mathcal A(\Lambda)\Sigma)\\
\text{s.t.}\quad&r\ge0,\qquad
\begin{pmatrix}
\Sigma&Y\\
Y^\top&\hat\Sigma
\end{pmatrix}\succeq0,\qquad
r+\tr(\Sigma+\hat\Sigma-2Y)\le\delta^2.
\end{aligned}
\end{equation}
Since $\delta>0$ and $\hat\Sigma\succ0$, taking $\Sigma=\hat\Sigma$, $Y=(1-\varepsilon)\hat\Sigma$, and $r=\delta^2/2$
 gives a Slater point for \eqref{eq:lifted_regret_primal} for sufficiently small $\varepsilon>0$.
Applying the conic duality theorem~\cite[Thm.~2.4.1(3.a)]{ben_tal_nemirovski2001} gives strong duality and attainment of the dual minimum, so $\mathcal R(\Lambda)$ equals
\begin{equation}\label{eq:drro_reduced_dual}
\begin{aligned}
\min_{\gamma,U}\quad&\gamma(\delta^2-\tr\hat\Sigma)+\tr(U\hat\Sigma)\\
\text{s.t.}\quad&\gamma I_d\succeq\mathcal B(\Lambda),\quad
\begin{pmatrix}
\gamma I_d-\mathcal A(\Lambda)&\gamma I_d\\
\gamma I_d&U
\end{pmatrix}\succeq0.
\end{aligned}
\end{equation}

To minimize $\mathcal R(\Lambda)$, we include $\Lambda$ among the optimization variables in \eqref{eq:drro_reduced_dual}.
The constraints are then no longer LMIs, since $\mathcal A(\Lambda)$ and $\mathcal B(\Lambda)$ depend quadratically on $\Lambda$.
To recover an LMI formulation, we introduce auxiliary variables $V_t,W_t\in\mathbb S^d$, $t=1,\dots,T-1$, satisfying
\begin{equation}\label{eq:drro_lifting_bounds}
V_t\succeq\frac{1}{t}\Lambda_t^\top M_t\Lambda_t,
\qquad
W_t\succeq(\Lambda_t-\bar H_t)^\top M_t(\Lambda_t-\bar H_t).
\end{equation}
We form a lifted problem by replacing $\mathcal A(\Lambda)$ by $\sum_{t=1}^{T-1}V_t$ and $\mathcal B(\Lambda)$ by $\bar H_0^\top M_0\bar H_0+\sum_{t=1}^{T-1}W_t$ in \eqref{eq:drro_reduced_dual}, subject to \eqref{eq:drro_lifting_bounds}.
Every feasible point of the lifted problem projects to a feasible $(\Lambda,\gamma,U)$ in \eqref{eq:drro_reduced_dual} with the same objective value.
Conversely, every feasible $(\Lambda,\gamma,U)$ extends to a feasible point of the lifted problem by choosing $V_t,W_t$ with equality in \eqref{eq:drro_lifting_bounds}, again with the same objective value.
Thus the lift is exact.
Since $M_t\succ0$, Schur complements convert the bounds in \eqref{eq:drro_lifting_bounds} into LMIs, yielding \eqref{eq:convex_row_sum_sdp}.

By \eqref{eq:regret_moment_representation}, $\mathcal R$ is lower semicontinuous and coercive, the latter because $\mathcal R(\Lambda)\ge\tr(\mathcal A(\Lambda)\hat\Sigma)\to\infty$ as $\sum_{t=1}^{T-1}\|\Lambda_t\|_F^2\to\infty$, since $M_t\succ0$ and $\hat\Sigma\succ0$.
Thus $\mathcal R$ attains its minimum at some $\Lambda^\star$.
Substituting $\tilde g_t=0$ and $\tilde F_{ts}=\Lambda_t^\star/t$ into \eqref{eq:policy_parametrization_new} yields the DRRO-optimal controller \eqref{eq:policy_representation_drro}.

\hfill\QED

\subsection{Proof of Theorem~\ref{thm:wc_recipe}}\label{app:worst_case_proof}
Fix an affine policy with residual coefficients $(\tilde g,\tilde F)$ from \eqref{eq:policy_parametrization_new}.
If $\alpha=\beta=0$, then $\mathcal A=\mathcal B=0$.
Since $M_t\succ0$, \eqref{eq:regret_coefficients} then gives $\tilde F=0$ and $\bar H_t=0$ for all $t$, so $c=0$.
Thus regret is constant at $a$, and every law in $\mathcal P_{\cW}$ is worst-case.

Now suppose $\max\{\alpha,\beta\}>0$.
For fixed feasible $\Sigma$ in \eqref{eq:controller_objective_in_Fg}, let $\rho=\delta^2-\mathsf B^2(\Sigma,\hat\Sigma)$.
Boyd and Vandenberghe~\cite[App.~B.1, Eqs.~(B.1)--(B.3)]{boyd_vandenberghe2004} show that, for $\rho>0$, the partial maximization over $z$ admits the representation
\[
\max_{\|z\|_2^2\le\rho}
\bigl\{z^\top\mathcal B(\tilde F)z+2c(\tilde g,\tilde F)^\top z\bigr\}
=\begin{aligned}[t]
\max_{z,\,Z\in\mathbb S^d}\quad&\tr(\mathcal B(\tilde F)Z)+2c(\tilde g,\tilde F)^\top z\\
\text{s.t.}\quad&\begin{pmatrix}Z&z\\z^\top&1\end{pmatrix}\succeq0,\quad \tr Z\le\rho.
\end{aligned}
\]
For $\rho=0$, both formulations force $z=0$, with $Z=0$ in the lift, so the equality still holds.
Combining this exact mean lift with the exact lift given by \eqref{eq:gelbrich_representation} rewrites \eqref{eq:controller_objective_in_Fg} as
\[
\begin{aligned}
\max_{z,Z,\Sigma,Y}\quad&
a(\tilde g,\tilde F)+\tr(\mathcal B(\tilde F)Z)+2c(\tilde g,\tilde F)^\top z+\tr(\mathcal A(\tilde F)\Sigma)\\
\text{s.t.}\quad&\begin{pmatrix}Z&z\\z^\top&1\end{pmatrix}\succeq0,\quad
\begin{pmatrix}\Sigma&Y\\Y^\top&\hat\Sigma\end{pmatrix}\succeq0,\qquad
\tr Z+\tr(\Sigma+\hat\Sigma-2Y)\leq\delta^2.
\end{aligned}
\]
Compactness of $\cW$ implies that the lifted optimum is finite.
The choice $z=0$, $Z=\varepsilon I_d$, $\Sigma=\hat\Sigma$, and $Y=(1-\varepsilon)\hat\Sigma$ is a Slater point for sufficiently small $\varepsilon>0$.
Conic duality~\cite[Thm.~2.4.1(3.a)]{ben_tal_nemirovski2001} therefore gives strong duality and dual attainment, so $\mathcal R(\tilde g,\tilde F)$ equals
\begin{equation}\label{eq:wc_regret_dual}
\begin{aligned}
\min_{\gamma,\tau,U}\quad&
a(\tilde g,\tilde F)+\gamma(\delta^2-\tr\hat\Sigma)+\tau+\tr(U\hat\Sigma)\\
\text{s.t.}\quad&\gamma\ge0,\quad \tau\in\R,\quad U\in\mathbb S^d,\quad
\begin{pmatrix}
\gamma I_d-\mathcal B(\tilde F)&c(\tilde g,\tilde F)\\
c(\tilde g,\tilde F)^\top&\tau
\end{pmatrix}\succeq0,\quad
\begin{pmatrix}
\gamma I_d-\mathcal A(\tilde F)&\gamma I_d\\
\gamma I_d&U
\end{pmatrix}\succeq0.
\end{aligned}
\end{equation}
The LMIs imply $\gamma\ge\max\{\alpha,\beta\}>0$.
Since the off-diagonal block $\gamma I_d$ is invertible, the second LMI forces $\gamma>\alpha$.
The generalized Schur complement of the first LMI requires $c(\tilde g,\tilde F)\in\operatorname{range}(\gamma I_d-\mathcal B(\tilde F))$.
Thus, the two LMIs are equivalent to $\gamma\in\mathcal D$,
\(\tau\ge c(\tilde g,\tilde F)^\top(\gamma I_d-\mathcal B(\tilde F))^\dagger c(\tilde g,\tilde F)\),
and \(U\succeq\gamma^2(\gamma I_d-\mathcal A(\tilde F))^{-1}\).
Since $\hat\Sigma\succ0$, minimizing over $(\tau,U)$ makes both bounds equalities.
Substitution into \eqref{eq:wc_regret_dual} gives $\mathcal R(\tilde g,\tilde F)=\min_{\gamma\in\mathcal D}\Psi(\gamma)$, with $\mathcal D$ and $\Psi$ as in \eqref{eq:wc_scalar_domain}--\eqref{eq:wc_scalar_objective}.
Using the definitions in \eqref{eq:wc_moment_maps}, differentiation on the interior of $\mathcal D$ gives
\begin{equation}\label{eq:wc_scalar_derivative}
\Psi'(\gamma)=\delta^2-\|z_\gamma\|_2^2
-\tr\!\bigl(\mathcal A(\tilde F)^2(\gamma I_d-\mathcal A(\tilde F))^{-2}\hat\Sigma\bigr)
=\delta^2-\|z_\gamma\|_2^2-\mathsf B^2(\Sigma_\gamma,\hat\Sigma),
\end{equation}
where the second equality follows from a direct computation.
The second derivative satisfies
\[
\Psi''(\gamma)=2c(\tilde g,\tilde F)^\top(\gamma I_d-\mathcal B(\tilde F))^{-3}c(\tilde g,\tilde F)
+2\tr\!\bigl(\mathcal A(\tilde F)^2(\gamma I_d-\mathcal A(\tilde F))^{-3}\hat\Sigma\bigr)\geq0.
\]
Equation~\eqref{eq:wc_scalar_derivative} also gives the right derivative at
$\gamma=\beta$ whenever $\beta\in \mathcal D$. Hence $\Psi$ is convex.
It is strictly convex unless $\mathcal A(\tilde F)=0$ and $c(\tilde g,\tilde F)=0$. In that case
$\Psi(\gamma)=a(\tilde g,\tilde F)+\gamma\delta^2$ and $\beta>0$, so its unique
minimizer is $\beta$.
In either case, $\Psi(\gamma)\to\infty$ as $\gamma\to\infty$ because $\delta>0$.
The behavior on the left boundary of $\mathcal D$ is one of the following:
\begin{itemize}
\setlength{\itemsep}{0pt}
\item If $\alpha\ge\beta$, then $\mathcal D=(\alpha,\infty)$ and $\Psi(\gamma)\to\infty$ as $\gamma\downarrow\alpha$, since $\hat\Sigma\succ0$.
\item If $\beta>\alpha$ and the range condition in \eqref{eq:wc_scalar_domain} fails at $\gamma=\beta$, then $\mathcal D=(\beta,\infty)$ and $\Psi(\gamma)\to\infty$ as $\gamma\downarrow\beta$.
\item If $\beta>\alpha$ and the range condition holds at $\gamma=\beta$, then $\mathcal D=[\beta,\infty)$ and $\Psi$ is continuous at $\beta$.
\end{itemize}
Thus $\Psi$ attains its minimum, and strict convexity makes the minimizer $\gamma^\star$ unique.
Write
\(
r^2:=
\delta^2-\|z_{\gamma^\star}\|_2^2
-\mathsf B^2(\Sigma_{\gamma^\star},\hat\Sigma).
\)
If $\gamma^\star>\beta$, then $\gamma^\star$ lies in the interior of $\mathcal D$, so first-order optimality and \eqref{eq:wc_scalar_derivative} give $r^2=\Psi'(\gamma^\star)=0$.
If $\gamma^\star=\beta$, optimality gives
$r^2=\Psi'_+(\beta)\geq0$.
We now use the dual optimizer $\gamma^\star$ to construct a moment pair attaining the maximum in \eqref{eq:controller_objective_in_Fg}.
Choose
\(v\in\ker(\gamma^\star I_d-\mathcal B(\tilde F))\) with \(\|v\|_2^2=r^2\)
and let
\(z^\star=z_{\gamma^\star}+v\) and \(\Sigma^\star=\Sigma_{\gamma^\star}\).
Set $\mu^\star=\hat\mu+z^\star$.
Since $z_{\gamma^\star}\perp\ker(\gamma^\star I_d-\mathcal B(\tilde F))$, the moment pair $(\mu^\star,\Sigma^\star)$ exhausts the ambiguity budget.
To verify that $(\mu^\star,\Sigma^\star)$ is worst-case, a calculation gives
\[
\begin{aligned}
&a(\tilde g,\tilde F)+(z^\star)^\top\mathcal B(\tilde F)z^\star+2c(\tilde g,\tilde F)^\top z^\star+\tr(\mathcal A(\tilde F)\Sigma^\star)\\
&=\Psi(\gamma^\star)-\gamma^\star\bigl[\delta^2-\|z^\star\|_2^2-\mathsf B^2(\Sigma^\star,\hat\Sigma)\bigr]=\Psi(\gamma^\star).
\end{aligned}
\]
Thus $\gamma^\star$ solves the scalar dual, and $(\mu^\star,\Sigma^\star)$ is feasible for \eqref{eq:controller_objective_in_Fg} with the same objective value.
Weak duality therefore shows that $(\mu^\star,\Sigma^\star)$ is a worst-case moment pair.
Conversely, for $\gamma\in\mathcal D$ and $(\hat\mu+z,\Sigma)\in\cW$, completing squares in $z$ and in the covariance lift \eqref{eq:gelbrich_representation} gives
\begin{equation}\label{eq:wc_component_bounds}
\begin{aligned}
z^\top\mathcal B(\tilde F)z+2c(\tilde g,\tilde F)^\top z
&\le\gamma\|z\|_2^2+c(\tilde g,\tilde F)^\top(\gamma I_d-\mathcal B(\tilde F))^\dagger c(\tilde g,\tilde F),\\
\tr(\mathcal A(\tilde F)\Sigma)
&\le\gamma\mathsf B^2(\Sigma,\hat\Sigma)-\gamma\tr\hat\Sigma
+\gamma^2\tr\!\bigl((\gamma I_d-\mathcal A(\tilde F))^{-1}\hat\Sigma\bigr).
\end{aligned}
\end{equation}
Equality holds in the first inequality if and only if $z-z_\gamma\in\ker(\gamma I_d-\mathcal B(\tilde F))$, and in the second if and only if $\Sigma=\Sigma_\gamma$.
Summing the bounds in \eqref{eq:wc_component_bounds}, adding $a(\tilde g,\tilde F)$ to both sides, and using the definition of $\Psi$ in \eqref{eq:wc_scalar_objective} gives
\[
a(\tilde g,\tilde F)+z^\top\mathcal B(\tilde F)z+2c(\tilde g,\tilde F)^\top z+\tr(\mathcal A(\tilde F)\Sigma)
\le\Psi(\gamma)-\gamma\bigl[\delta^2-\|z\|_2^2-\mathsf B^2(\Sigma,\hat\Sigma)\bigr]\le\Psi(\gamma).
\]
Since $\gamma^\star>0$, attaining $\Psi(\gamma^\star)$ requires equality in both bounds and exhaustion of the ambiguity budget.
Thus all worst-case moment pairs are exactly $(\hat\mu+z_{\gamma^\star}+v,\Sigma_{\gamma^\star})$ with $v$ satisfying \eqref{eq:wc_moment_recovery}. In particular, the worst-case
pair is unique exactly when $r^2=0$. Thus, multiple worst-case regret moment pairs
can only occur when $\gamma^\star=\beta$.

We conclude by showing that in the DRRO case there are indeed at least two worst-case moment pairs:
Let $(\tilde g^\star,\tilde F^\star)$ represent a DRRO-optimal policy under the parametrization \eqref{eq:policy_parametrization_new}.
For each fixed feasible moment pair $(\hat\mu+z,\Sigma)$, let $f_{z,\Sigma}(\tilde g,\tilde F)$ denote the objective inside the maximization in \eqref{eq:controller_objective_in_Fg}.
This is a convex differentiable quadratic in $(\tilde g, \tilde F)$ with minimum zero, attained at $\tilde F=0$ and $\tilde g_t=\bar H_tz$, corresponding to the oracle controller.
Suppose the worst-case moment pair $(\hat\mu+z^\star,\Sigma^\star)$ is unique.
Then
\(
\nabla f_{z^\star,\Sigma^\star}(\tilde g^\star,\tilde F^\star)
=\nabla\mathcal R(\tilde g^\star,\tilde F^\star)=0.
\)
The first equality follows from Danskin's theorem, using compactness of $\cW$, joint continuity of $f_{z,\Sigma}$ and its gradient in $(\tilde g,\tilde F)$, and the assumed uniqueness of the maximizer.
The second follows from first-order optimality of $(\tilde g^\star,\tilde F^\star)$ as an unconstrained minimizer of $\mathcal R$.
Since $f_{z^\star,\Sigma^\star}$ is convex with minimum zero, the vanishing gradient implies $f_{z^\star,\Sigma^\star}(\tilde g^\star,\tilde F^\star)=0$.
Because $(\hat\mu+z^\star,\Sigma^\star)$ is worst-case, it follows that
\(
\mathcal R(\tilde g^\star,\tilde F^\star)
=f_{z^\star,\Sigma^\star}(\tilde g^\star,\tilde F^\star)=0.
\)
For every $(\hat\mu+z,\Sigma)\in\cW$, nonnegativity of regret therefore gives
\(
0\le f_{z,\Sigma}(\tilde g^\star,\tilde F^\star)
\le\mathcal R(\tilde g^\star,\tilde F^\star)=0.
\)
Thus every pair in $\cW$ attains the worst-case value, contradicting uniqueness since $\delta>0$ ensures that $\cW$ contains more than one pair.
Hence every DRRO-optimal affine policy admits at least two worst-case moment pairs.
For such a policy with $\max\{\alpha,\beta\}>0$, \eqref{eq:wc_scalar_domain} and \eqref{eq:wc_moment_recovery} yield
$r^2>0$ and $\gamma^\star=\beta>\alpha$.\hfill\QED

\bibliographystyle{IEEEtran}
\bibliography{references}

\end{document}